\documentclass[11pt]{article}

\usepackage{amsmath}
\usepackage{amssymb}
\usepackage{amsthm}
\usepackage{comment}
\usepackage{enumitem}
\usepackage{ifthen}
\usepackage{mathrsfs}
\usepackage{mathabx}
\usepackage{mathtools}\usepackage[USenglish]{babel}
\usepackage{wrapfig}
\usepackage{graphicx}
\usepackage[margin=1.03in]{geometry}
\usepackage{subcaption}
\usepackage[table]{xcolor}
\usepackage{makecell}
\usepackage{enumitem}
\usepackage{bm}
\usepackage{url}
\usepackage{float}
\usepackage{dsfont}
\usepackage{color} 
\usepackage{csquotes}

\makeatletter
\let\cite\relax
\DeclareRobustCommand{\cite}{%
  \let\new@cite@pre\@gobble
  \@ifnextchar[\new@cite{\@citex[]}}
\def\new@cite[#1]{\@ifnextchar[{\new@citea{#1}}{\@citex[#1]}}
\def\new@citea#1{\def\new@cite@pre{#1}\@citex}
\def\@cite#1#2{[{\new@cite@pre\space#1\if\relax\detokenize{#2}\relax\else, #2\fi}]}
\makeatother

\usepackage{titlesec}
\titlespacing*{\section}{0pt}{6pt}{2pt}
\titlespacing*{\subsection}{0pt}{6pt}{6pt}
\titlespacing*{\subsubsection}{0pt}{6pt}{0pt}
\titlespacing*{\figure}{0pt}{6pt}{0pt}
\titlespacing*{\table}{0pt}{6pt}{0pt}

\newtheorem{theorem}{Theorem}[section]
\newtheorem{lemma}[theorem]{Lemma}
\newtheorem{corollary}{Corollary}[theorem]
\newtheorem{proposition}[theorem]{Proposition}
\newtheorem{definition}[theorem]{Definition}
\newtheorem{remark}[theorem]{Remark}
\numberwithin{equation}{section}

\begin{document}

\title{A modified chemostat exhibiting competitive exclusion ``reversal"}

\author{\centerline{\scshape  Thomas Griffin$^{1}$, James Lathrop$^{2}$ and  Rana D. Parshad$^{1}$}}

\maketitle

\medskip
{\footnotesize

   \medskip
   
    \centerline{ 1) Department of Mathematics,}
 \centerline{Iowa State University,}
   \centerline{Ames, IA 50011, USA.}

    \medskip
   
    \centerline{ 2) Department of Computer Science,}
 \centerline{Iowa State University,}
   \centerline{Ames, IA 50011, USA.}

 }

\medskip

\begin{abstract}

The classical chemostat is an intensely investigated model in ecology and bio/chemical engineering, where n-species, say $x_{1}, x_{2}...x_{n}$ compete for a single growth limiting nutrient. Classical theory predicts that depending on model parameters, one species competitively excludes all others. Furthermore, this ``order" of strongest to weakest is preserved, $x_{1} >> x_{2} >> ...x_{n}$, for say $D_{1} < D_{2} <...D_{n}$, where $D_{i}$ is the net removal of species $x_{i}$.
Meaning $x_{1}$ is the strongest or most dominant species and $x_{n}$ is the weakest or least dominant. We propose a modified version of the classical chemostat, exhibiting certain counterintuitive dynamics. Herein we show that if only a certain proportion of the weakest species $x_{n}$'s population is \emph{removed} at a ``very" fast density dependent rate, it will in fact be able to competitively exclude all other species, for certain initial conditions.
Numerical simulations are carried out to visualize these dynamics in the three species case. 

\end{abstract}

    \maketitle
    \section{Introduction}
\label{sec:introduction}

Chemostat is an intriguing model that sits at the intersection of biology, chemistry, and engineering.
It is a controlled environment where microorganisms engage in a dynamic dance of growth, competition, and survival, all orchestrated by carefully calibrated inputs and outputs. 
The chemostat, short for "chemical environment station," is a powerful tool used in microbiology and bio-process engineering to study the interactions between microbial populations and their environment.
With its ability to mimic real-world ecosystems and sustain continuous cultures under precise conditions, the chemostat opens a window into understanding fundamental biological processes, optimizing industrial production of biofuels, pharmaceuticals, and other valuable compounds, and even shedding light on ecological dynamics \cite{calcott2020continuous}.

One of the most remarkable aspects of the chemostat is its ability to maintain a specific dilution rate, which refers to the rate at which fresh medium is introduced and old medium is removed. 
This rate determines the time microorganisms spend within the system and plays a crucial role in shaping their growth rates and interactions. 
By adjusting the dilution rate, researchers can simulate scenarios ranging from rapid population growth to resource limitation, enabling the study of various ecological and physiological responses \cite{ajbar2011dynamics} \cite{degermendzhi2023control}.

In particular the general chemostat ODE model is an expansion of the classical Lotka-Volterra competition model \cite{vet2018bistability}. Unlike Lotka-Volterra  models, chemostat models consider nutrient concentrations explicitly. As a result, they can encompass more intricate interactions, such as the interplay between mutualism and competition. In this context, we demonstrate the ability to simplify the relatively complex chemostat equations into a pair of extended Lotka-Volterra equations.

In this paper we will investigate a novel change in the classical chemostat model, first demonstrated with the two species case and then further extended to the n-species, and outline the dynamics of the new system. Also we will showcase many numerical results as well as explore the impact of changing the human controlled parameters of the species inside the system.

    \section{Background}
\label{sec:background}

Consider the classical Chemostat system, 

\begin{equation}
\frac{dx_{i}}{dt} = x_{i}(f_{i}(S)-D_{i}), \ \frac{dS}{dt} = D(S^{0} - S) - \sum^{n}_{i=1} \frac{1}{\gamma_{i}} f_{i}(S)x_{i}, \ i=1,2...,n.
\end{equation}

This system describes the competition of the $x_{i}$ species with a growth-limiting resource/nutrient $S$.
The $f_{i}$ are growth functions of the species that are assumed to be monotonically increasing.
The yield constants, denoted by $\gamma_{i}$, account for the fraction of the nutrient that the different species consume that is used to create new biomass.
$D$ is the resources dilution rate, $S^{0}$ is its concentration limit, and $D_{i} = D+k_{i}$ (where $k_{i}$ is typically a death rate) are removal rates for the competing species. 
If the parameters are kept constant, then at most one species survives as dictated by \emph{principle of competitive exclusion} \cite{hardin1960competitive} \cite{liu2014competitive}.
Since co-existence has been observed in many natural systems, mechanisms to achieve this result mathematically, have been thoroughly studied \cite{lotka1925elements} \cite{volterra1926fluctuations} \cite{armstrong1976coexistence}
\cite{sanders1995odum}.
To implement coexistence among the competing organisms within the chemostat a significant body of literature has focused on refining the chemostat model through the application of control theory.
This includes employing open-loop control, as demonstrated in \cite{harmand2017chemostat} \cite{smith1981competitive} and feedback control \cite{de2002feedback} \cite{gouze2005feedback} \cite{keeran2008feedback}.
Specifically, De Leenheer and Smith \cite{de2002feedback}, along with Keeran et al. \cite{keeran2008feedback}, regarded the dilution rate as a feedback parameter, linking it to the concentrations of the competing organisms. This approach is grounded in the understanding that the dilution rate functions as a lab-controlled variable, and the concentrations of organisms are measurable.
Reference \cite{de2002feedback} establishes that when the dilution rate is an affine function dependent on the concentrations of two competing organisms, coexistence can be attained as a globally asymptotically stable positive equilibrium.
Furthermore, \cite{keeran2008feedback} illustrates that within a chemostat featuring two competing organisms and governed by a novel form of feedback control, the presence of an asymptotically stable positive limit cycle is possible.
In this paper, we will investigate a density dependent removal rate that when applied can lead to instances where the weaker/weakest competitor could win out and competitively exclude all the other competitors depending on initial conditions. 
    \section{Two species extraction model}

We introduce the following two species chemostat extraction model, where the weaker species is harvested at a density  dependent rate. Intuitively one would suspect that if we harvested the weakest species that the extinction would of that species would still happen, perhaps at an even faster rate, however as we will demonstrate under certain conditions we can actually see the weaker species outlast its stronger competitors. Similar work has been done when non-smooth density dependent rates have been considered in both predator-prey and competitive scenarios \cite{antwi2020dynamics, parshad2021some}. The current model takes the form,

\begin{align}
\frac{dx_{1}}{dt}&= x_{1}(f_{1}(S)-D_{1}),\\
\frac{dx_{2}}{dt}&= x_{2}(f_{2}(S)-qD_{2})-(1-q)kx_{2}^{p}, \text{ and,}  \\
\frac{dS}{dt}&= D(S^{0}-S) - \sum^{2}_{i=1} \frac{1}{\gamma_{i}} f_{i}(S)x_{i}
\end{align}
where $q$ is the percentage of the species that are removed at a density dependent rate $k(x_{2}) = k x_{2}^{p-1}$.
We will consider functional responses of the form,
\begin{equation}
f_{i}(S)=\frac{m_{i}S}{S+a_{i}}. 
\end{equation}

\begin{lemma}
    
The equilibrium point $E^{x_2}=(0,x_2^{*},S^{*})$ is a local sink equilibrium when,

$$(f_2-qD_2)(1-p)-\frac{f_2}{\gamma_{2}}  + \frac{a_2 m_2}{\gamma_2 (S^{*}+a_2)^{2}}  x_{2}^{*} < D$$ 
and
$$D < (\frac{a_2 m_2}{\gamma_2 (S^{*}+a_2)^{2}}  x_{2}^{*}) + (\frac{f_2}{\gamma_{2}}(f_2-qD_2)(1-p))(x_2^* \,{\left(\frac{m_2}{S^*+a_2}-\frac{S^*\,m_2 }{{{\left(S^*+a_2 \right)}}^2 }\right)})$$
\end{lemma}

\begin{proof}

We will begin by looking at the Jacobian with the equilibrium plugged in and then looking at the eigenvalues of the system.

\begin{equation*}
J(E^{x_2})=\begin{bmatrix}
f_1 - D_{1} & 0 & 0 \\

0 & (f_2-qD_2)(1-p) & x_2^* \,{\left(\frac{m_2}{S^*+a_2}-\frac{S^*\,m_2}{{{\left(S^*+a_2 \right)}}^2 }\right)}\\

-\frac{f_1 }{\gamma_{1}}   & -\frac{f_2 }{\gamma_{2}}  & -D + \frac{a_2 m_2}{\gamma_2 (S^{*}+a_2)^{2}}  x_{2}^{*}
\end{bmatrix}
\end{equation*}

Characteristic polynomial:

$$\rho(\lambda) = (f_1-D_1-\lambda) \det \left( \begin{bmatrix}

(f_2-qD_2)(1-p) - \lambda & x_2^* \,{\left(\frac{m_2}{S^*+a_2}-\frac{S^*\,m_2}{{{\left(S^*+a_2 \right)}}^2 }\right)}\\
-\frac{f_2}{\gamma_{2}} & -D + \frac{a_2 m_2}{\gamma_2 (S^{*}+a_2)^{2}}  x_{2}^{*} - \lambda
\end{bmatrix}
\right)$$

Looking at the eigenvalues for this system, the first is $\lambda=f_1-D_1$, which by construction must be less than or equal to zero. 
In the case where the eigenvalue is positive, then the differentials of the $x_1$ equation would be positive and have growth thus $E^{x_{2}}$ would not be an equilibrium point.

The remaining two eigenvalues are determined from the bottom right 2x2 matrix, and to achieve the condition of local sink, we need both eigenvalues to be negative.

When dealing with the final 2x2 matrix we can determine the signs of the eigenvalues using the trace and determinant. 
The determinant is product of the eigenvalues and when positive either means that both eigenvalues are positive or both are negative.
The trace is the sum of the eigenvalues, if the trace is negative coupled with the positive determinant implies that both eigenvalues are negative as desired.

This trace condition holds when:

$$ (f_2-qD_2)(1-p)-\frac{f_2}{\gamma_{2}} -D + \frac{a_2 m_2}{\gamma_2 (S^{*}+a_2)^{2}}  x_{2}^{*}<0$$

$$ (f_2-qD_2)(1-p)-\frac{f_2}{\gamma_{2}}  + \frac{a_2 m_2}{\gamma_2 (S^{*}+a_2)^{2}}  x_{2}^{*}<D$$

The determinant condition holds when:

$$(f_2-qD_2)(1-p)(-D + \frac{a_2 m_2}{\gamma_2 (S^{*}+a_2)^{2}}  x_{2}^{*}) - (-\frac{f_2}{\gamma_{2}})(x_2^* \,{\left(\frac{m_2}{S^*+a_2}-\frac{S^*\,m_2 }{{{\left(S^*+a_2 \right)}}^2 }\right)}) > 0$$

$$(\frac{a_2 m_2}{\gamma_2 (S^{*}+a_2 )^{2}}  x_{2}^{*}) + (\frac{f_2}{\gamma_{2}}(f_2-qD_2)(1-p))(x_2^* \,{\left(\frac{m_2 }{S^*+a_2 }-\frac{S^*\,m_2 }{{{\left(S^*+a_2 \right)}}^2 }\right)}) > D$$
  
\end{proof}

\section{Reduction of the Chemostat model to Lotka-Volterra}

Setting S as some limiting $S^{*}$ we get the following reduction of the 2 species extraction chemostat to this Lotka-Volterra competition model.

\begin{align}
\frac{dx_{1}}{dt}&= x_{1}((D(S^{0}-S^{*})-f_{1}(S^{*})x_{1}-f_{2}(S^{*})x_{2})-D_{1}),\\
\frac{dx_{2}}{dt}&= x_{2}((D(S^{0}-S^{*})-f_{1}(S^{*})x_{1}-f_{2}(S^{*})x_{2})-qD_{2})-(1-q)kx_{2}^{p}
\end{align}

Our goal here is to reduce the chemostat model to a model of only two equations with reduced complexity, allowing for a qualitative dynamical analysis, while retaining a physical interpretation of each term in the model. 
This model reduction consists of the following critical steps:

1.	Elimination of the nutrient variable S, which is possible since linear relations between S and $x_{i}$ exist as t goes to infinity.

2.	Simplification of the growth rates using a Taylor approximation when $a_{i} >>$  S. It is always possible to find a parameter set for which the system has the same steady states that obeys this condition. 

3.	Elimination of higher order terms that are not critical for the qualitative dynamics.

For full derivation please see \cite{vet2018bistability}.

After completing these steps, we are left with the following system.

\begin{align}
\label{eq:slv}
\frac{dx_{1}}{dt}&= x_{1}(a-b_{1}x_{1}-c_{1}x_{2}-D_{1}),\\
\frac{dx_{2}}{dt}&= x_{2}(a-c_{2}x_{1}-b_{2}x_{2}-qD_{2})-(1-q)kx_{2}^{p}
\end{align}

\begin{figure}[H]
\begin{subfigure}[b]{.475\linewidth}
    \includegraphics[width=3.0in]{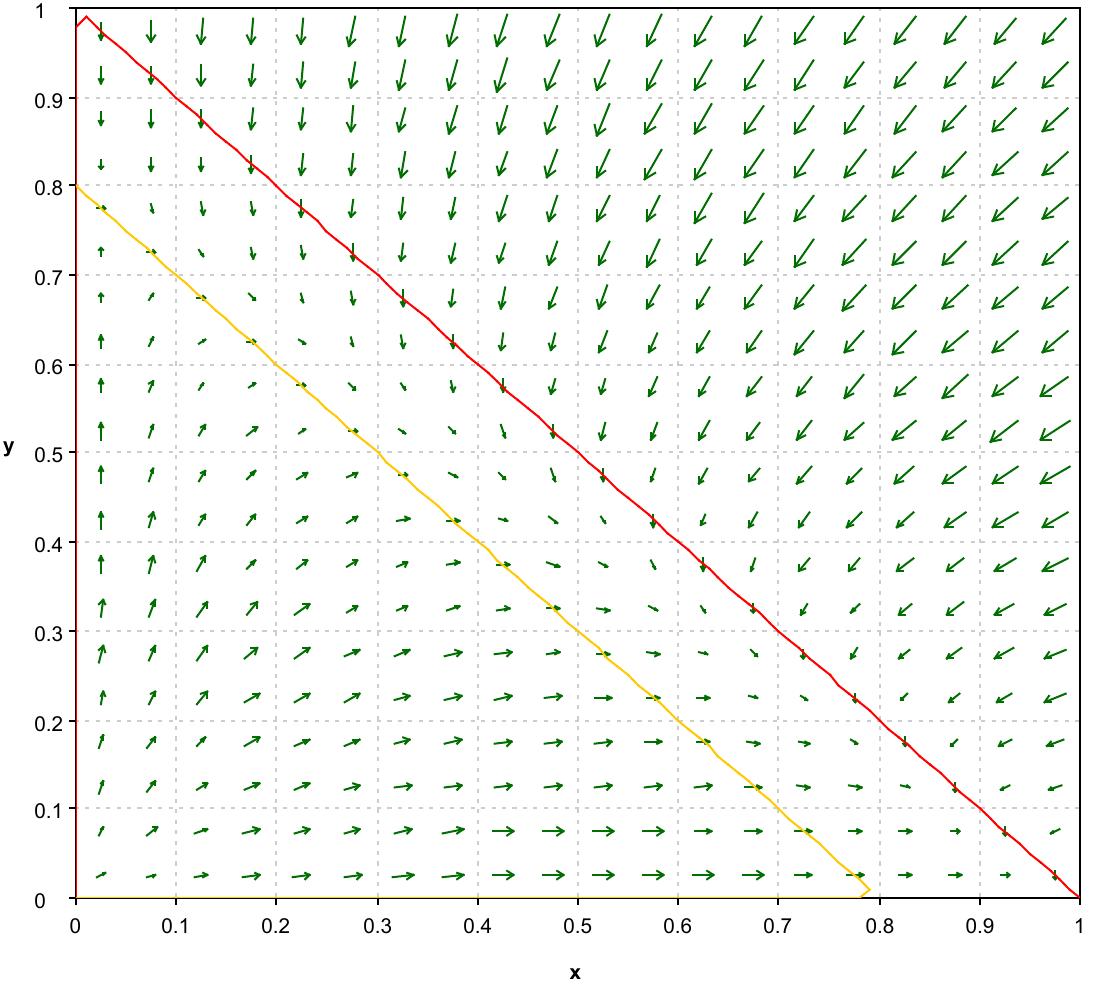}
    \caption{$q=1$}
\end{subfigure}
\hfill
\begin{subfigure}[b]{.475\linewidth}
    \includegraphics[width=3.0in]{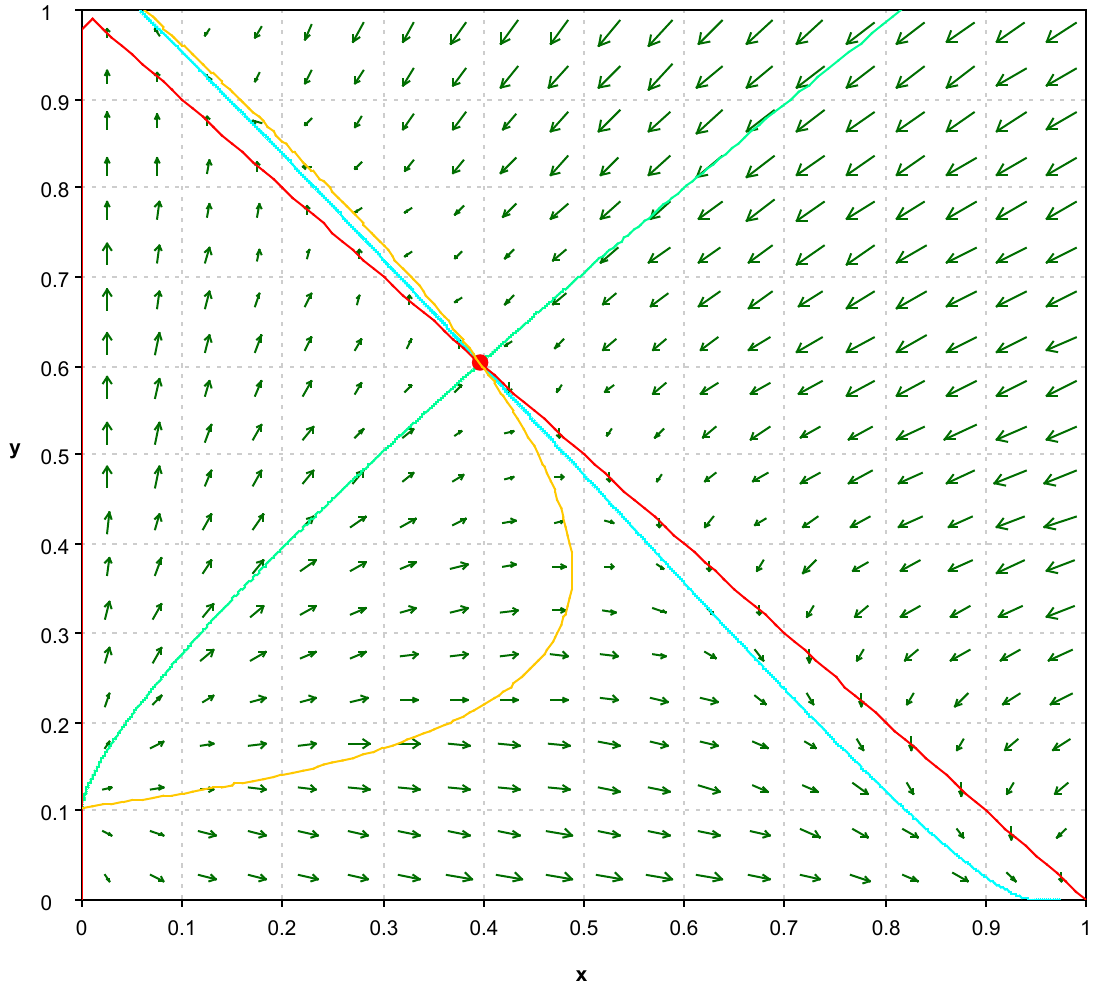}
    \caption{$q=.8$}
\end{subfigure}
\caption{For a set of parameter values we can see that when $q=1$ (no percentage of the weaker $x_{2}$ species is affected by the removal term) that species $x_{1}$ competitively excludes $x_{2}$ for all positive initial conditions. 
As q is decreased and the percentage of the $x_{2}$ extracted increases. 
From this a seperatrix of the initial conditions is created leading to different dynamics depending on the starting location.
In (b) an initial condition above and to the left of the green seperatrix line will converge to only the weaker species surviving, while initial conditions below and to the right will converge to the stronger species excluding all others.}
\end{figure}

We first recap the following classical result  which guarantees non-negativity of solutions \cite{pierre2010global, perko2013differential}.

\begin{lemma}
\label{lem:l1}
Consider the following general ODE system,
\begin{eqnarray}
\label{C1} \dot{f} &=&  F(f,m),\\
\label{C2} \dot{m} &=&  G(f,m),\\
\end{eqnarray}
where $F, G$ are Lipschitz continuous in the state variables, $f,m$. Then the non-negativity of solutions is preserved in time, that is,
\begin{equation*}
f_{0},m_{0} \geq 0~~ \Rightarrow~~ \forall t \in [0, T_{max}),~ f(t),m(t) \geq 0,
\end{equation*}
if and only if
\begin{equation*}
\forall f(t),m(t) \geq 0 => F(0,m),~ G(f,0), \geq 0.
\end{equation*}

\end{lemma}

\begin{lemma}
    Consider the competition system \eqref{eq:slv}. Solutions to the system remain non-negative, for initial data $(x_{1}(0),x_{2}(0))$ that is positive.
\end{lemma}

\begin{proof}
    The proof follows via quasi-positivity of the system via application of lemma \ref{lem:l1}
\end{proof}

\begin{remark}
Note, the application of lemma \ref{lem:l1}, requires Lipschitz continuity of the R.H.S of the ODE in question. However, the R.H.S of \eqref{eq:slv} is not Lipschitz continuous. This suffices here as the Lipschitz continuity is required for uniqueness of solutions, which we are not concerned with in this setting. The quasi-positivity can be still used to obtain non-negativity of solutions.
\end{remark}

We introduce certain requisite definitions,

Given a system of ODE, depending on the non-linearity in the equations,
one might not expect a solution to always exist globally in time. In
particular, solutions of some ODE may blow up in finite time \cite{quittner2019superlinear}. Recall,

\begin{definition}
(Finite time blow-up for ODE) We say that a solution $u(t)$, of a given ODE, with
suitable initial conditions, blows up at a finite time if%
\begin{equation*}
\underset{t \rightarrow T^{\ast } < \infty}{\lim }\left\vert u\left( t\right)
\right\vert =+\infty ,
\end{equation*}%
where $T^{\ast } < \infty$ is the blow-up time.
\end{definition}

\begin{definition}
(Finite time extinction for ODE) We say that a solution $v(t)$, of a given ODE, with
suitable initial conditions, goes extinct at a finite time if%
\begin{equation*}
\underset{t \rightarrow T^{\ast } < \infty}{\lim }\left\vert v\left( t\right)
\right\vert =0 ,
\end{equation*}%
where $T^{\ast } < \infty$ is the extinction time.
\end{definition}

\bigskip We need the following alternative (See \cite{pierre2010global}).

\begin{proposition}
\label{prop.exis.blow.2.2} The system \eqref{eq:slv}
admits a unique local in time, classical solution $\left( x_{1}, x_{2} \right) $
on an interval $[0,T_{\max }]$, and either\newline
(i) \textit{The solution} \textit{is bounded on} $[0,T_{\max })$ \textit{and
it is global ( i.e.} $T_{\max }=+\infty $).\newline
\textit{(ii) Or }%
\begin{equation}
\underset{t\nearrow T_{\max }}{\lim }\max \hspace{.02in} \left\vert x_{1}(t)\right\vert
+\left\vert x_{2}(t)\right\vert  =+\infty ,
\label{2.D.C.R.2.3}
\end{equation}%
\textit{in this case, the solution is not global, and we say that it blows up in finite
time} $T_{\max }$, \textit{or it ceases to exist}, where $T_{\max } < \infty$
denotes the eventual blowing-up time.
\end{proposition}

\begin{theorem}
\label{thm:t11}
    Consider the competition system \eqref{eq:slv}. There exists initial data $(x_{1}(0),x_{2}(0))$ such that initiating from this data $x_{2}$ goes extinct in finite time, and $x_{1} \rightarrow \left(\frac{a-D_{1}}{b_{1}}\right)$ as $t \rightarrow \infty$.
\end{theorem}

\begin{proof}
   We follow methods in \cite{parshad2020remark, antwi2020dynamics} and begin by making a change of variable, $x_{2}=\frac{1}{u}$. Now,
   \begin{equation}
   \frac{dx_{2}}{dt} = \frac{-1}{u^{2}}\frac{du}{dt}
   \end{equation}
   Inserting the $u$ variable into the $x_{2}$ equation yields,

 \begin{equation}
   \frac{-1}{u^{2}}\frac{du}{dt} = \left(\frac{1}{u} \right)
   \left( a-c_{2}x_{1}-b_{2}\frac{1}{u}-qD_{2} -(1-q)k\left(\frac{1}{u} \right)^{p}\right)
   \end{equation}

Multiplying both sides by $-u^{2}$ and recasting the complete system in terms of the new variables we obtain,

\begin{align}
\label{eq:slv1}
\frac{dx_{1}}{dt}&= x_{1}(a-b_{1}x_{1}-c_{1}x_{2}-D_{1}),\\
\frac{du}{dt}&= c_{2}x_{1}u + b_{2} - (a-qD_{2})u + (1-q)ku^{2-p}
\end{align}
   Note,

\begin{equation}
  \frac{du}{dt}= c_{2}x_{1}u + b_{2} - (a-qD_{2})u + (1-q)ku^{2-p} > (1-q)ku^{2-p} - (a-qD_{2})u
   \end{equation}
Thus using positivity, and the fact that $0<p<1$, so $1<2-p<2$,
we see that $u$ will blow up in finite time, for sufficiently large positive initial data, in comparison with the ODE,
$\dot{x}=C_{1}x^{l}-C_{2}x, C_{1},C_{2}>0, 1<l<2$.
Thus, $\lim_{t \rightarrow T^{*} < \infty} u(t) \rightarrow + \infty$. And so,  

\begin{equation}
\lim_{t \rightarrow T^{*} < \infty} x_{2}(t) = \lim_{t \rightarrow T^{*} < \infty} \left(\frac{1}{u(t)} \right) = \frac{1}{(\lim_{t \rightarrow T^{*} < \infty} u(t))}\rightarrow \frac{1}{+ \infty} = 0
\end{equation}

Once we have $\lim_{t \rightarrow T^{*} < \infty} x_{2}(t) =0$, plugging that into the equation for $x_{1}$ we have,

\begin{equation}
  \frac{dx_{1}}{dt}= x_{1}(a-b_{1}x_{1}-D_{1})
   \end{equation}

   Thus via standard theory, $x_{1} \rightarrow \left(\frac{a-D_{1}}{b_{1}}\right)$ as $t \rightarrow \infty$. This proves the Theorem.

\end{proof}
   
\begin{remark}
We see in the proof of Theorem \ref{thm:t11} that we need sufficiently large initial data for $u$ to blow-up in finite time. Since 
$x_{2} = \frac{1}{u}$, large initial data for $u$ is tantamount to small initial data for $x_{2}$ to go extinct in finite time.
\end{remark}

    \section{Chemostat Extraction Model}

In this section we will introduce the Chemostat Extraction Model (CEM) defined by the following equations,
\begin{align}
\frac{dx_{i}}{dt}&= x_{i}(f_{i}(S)-D_{i}), \label{eq:chemostat2-1} \\
\frac{dx_{n}}{dt}&= x_{n}(f_{n}(S)-qD_{n})-(1-q)kx_{n}^{p}, \text{ and,}  \label{eq:chemostat2-2} \\
\frac{dS}{dt}&= D(S^{0}-S) - \sum^{n}_{i=1} \frac{1}{\gamma_{i}} f_{i}(S)x_{i}i=1,2...,n-1, \label{eq:chemostat2-3}
\end{align}
where $q$ is the percentage of the species that are removed by the density dependent removal controlled by rate k and power p.

We will consider the $n$ species case, with similar functional responses,
\begin{equation}
f_{i}(S)=\frac{m_{i}S}{S+a_{i}}. \label{eq:functional_response}
\end{equation}
We first study the dynamics of the interior equilibria, and
examine the stability of an interior equilibrium for the system defined by Equations~\ref{eq:chemostat2-1}, \ref{eq:chemostat2-2}, and \ref{eq:chemostat2-3}.

\subsection{Stability Analysis}

To fully understand the dynamics of this system we will conduct analysis of equilibrium solutions. 
This will show us the criteria needed to allow convergence to the desired equilibrium. 
We will begin by showing that there is no stable equilibrium where more than one of the competing species can survive. 

\begin{theorem}
An interior equilibrium $E^{*}=(x_1^{*},x_2^{*},...,S^{*})$ of the system defined by Equations~\ref{eq:chemostat2-1}, \ref{eq:chemostat2-2}, and \ref{eq:chemostat2-3} is a saddle equilibrium.
\end{theorem}
\begin{proof}
We first examine the nullclines of the system of equations given by 
\begin{align}
\label{chemo_nullclines}
 &f_{i}(S)-D_{i}=0, i=1,2,...,n-1, \\ 
 &f_{n}(S)-qD_{n}-(1-q)k (x_{n})^{p-1}=0, \text{ and} \\
 &D(S^{0} - S) - \sum^{n}_{i=1}f_{i}(S)x_{i}=0.
\end{align}

Solving the system of equations, without loss of generality let $E^{*}(x_{1}^{*},x_{2}^{*},...,S^{*})$ be a positive interior equilibria of system.
We study the $n+1$ by $n+1$ Jacobian of the point which is given by

\begin{equation*}
J(E^{*})=\begin{bmatrix}
f_1 - D_{1} & 0 & \dots  & 0 & x_1^* \,{\left(\frac{m_1 }{S^*+a_1 }-\frac{S^*\,m_1 }{{{\left(S^*+a_1 \right)}}^2 }\right)} \\

0 & f_2-D_2 & \ddots & \vdots & x_2^* \,{\left(\frac{m_2 }{S^*+a_2 }-\frac{S^*\,m_2 }{{{\left(S^*+a_2 \right)}}^2 }\right)}\\

\vdots & \vdots & \ddots & \vdots & \vdots\\

0 & 0 & \dots & f_n-qD_n - (1-q)kp(x_{n}^{*})^{p-1} & x_n^* \,{\left(\frac{m_n }{S^*+a_n }-\frac{S^*\,m_n }{{{\left(S^*+a_n \right)}}^2 }\right)}\\

-\frac{f_1 }{\gamma_{1}}   & -\frac{f_2 }{\gamma_{2}}  & \dots & -\frac{f_n}{\gamma_{n}} & -D - \sum^{n}_{i=1} \frac{a_i m_i}{(S^{*}+a_i)^{2}}  x_{i}^{*}
\end{bmatrix}
\end{equation*}

Simplifying with the aforementioned nullclines,

\begin{equation*}
J(E^{*})=\begin{bmatrix}
0 & 0 & \dots  & 0 & x_1^* \,{\left(\frac{m_1 }{S^*+a_1 }-\frac{S^*\,m_1 }{{{\left(S^*+a_1 \right)}}^2 }\right)} \\

0 & 0 & \ddots & \vdots & x_2^* \,{\left(\frac{m_2 }{S^*+a_2 }-\frac{S^*\,m_2 }{{{\left(S^*+a_2 \right)}}^2 }\right)}\\

\vdots & \vdots & \ddots & \vdots & \vdots\\

0 & 0 & \dots &  (1-p)(1-q)k(x_{n}^{*})^{p-1} & x_n^* \,{\left(\frac{m_n }{S^*+a_n }-\frac{S^*\,m_n }{{{\left(S^*+a_n \right)}}^2 }\right)}\\

-\frac{f_1 }{\gamma_{1}}   & -\frac{f_2 }{\gamma_{2}}  & \dots & -\frac{f_n}{\gamma_{n}} & -D - \sum^{n}_{i=1} \frac{a_i m_i}{(S^{*}+a_i)^{2}}  x_{i}^{*}
\end{bmatrix}
\end{equation*}

It is clear that the system has $\lambda=0$ as $n-1$ of the eigenvalues. 
Thus, when an interior equilibrium exists it is always an unstable saddle.
\end{proof}

\begin{remark}
Note due to the eigenvalues being zero, one has to deal with a degenerate saddle case. Also, the interior equilibrium may not be positive. In the event it is negative, due to its saddle nature we still have via standard theory the existence of a stable manifold, a $C^{1}$ hyper surface that separates the phase space, into different regions of attraction depending on initial conditions, \cite{perko2013differential}. 
    \end{remark}

\begin{lemma}
If $p=1$ then the system will be in competitive exclusion and the equilibrium point $E^{x_1}=(x_1^{*},0,...,0,S^{*})$ will be a sink if species $x_{1}$ is assumed the strongest competitor.
\end{lemma}

\begin{proof}
While proved fully in \cite{rapaport2017new} we will discuss the key points of the proof.

When the power $p=1$ the system reduces to 
\begin{align}
\frac{dx_{i}}{dt}&= x_{i}(f_{i}(S)-D_{i}),  \\
\frac{dx_{n}}{dt}&= x_{n}(f_{n}(S)-qD_{n}-(1-q)k), \text{ and,}  \\
\frac{dS}{dt}&= D(S^{0}-S) - \sum^{n}_{i=1} \frac{1}{\gamma_{i}} f_{i}(S)x_{i}, \ i=1,2...,n-1, 
\end{align}

Without loss of generality, relabel $-qD_{n}-(1-q)k$ to some constant $D_{n}$, it now has the form of the other $x_{i}$ equations.
With the assumptions that $f_{i}(S)$ has the qualities of being one-time differentiable (in $C^{1}$), $f_{i}(0)=0$, and $f'_{i}(S) > 0$ for $S>0$ (increasing).
When these assumptions are met each $x_{i}$ has a ``break-even point" $\Gamma_{i}$ which is defined as the unique solution to $f_{i}(\Gamma_{i})=D_{i}$. 
Without loss of generality let $\Gamma_{1} < \Gamma_{2} < ... < \Gamma_{n}$, with this we can show that each $x_{i}, i>1$ will converge to zero and only the species with the smallest break-even point will survive and then converge to the requisite equilibrium.

\end{proof}

\begin{remark}
    It is also known that in the biologically irrelevant case where some of the break-even points are equal then the species whose $\Gamma_{i}$ value are smallest will be the surviving species \cite{rapaport2017new}.
\end{remark}

\begin{lemma}
    
The equilibrium point $E^{x_n}=(0,0,...,x_n^{*},S^{*})$ is a local sink equilibrium when,

$$(f_n-qD_n)(1-p)-\frac{f_n}{\gamma_{n}}  + \frac{a_n m_n}{\gamma_n (S^{*}+a_n)^{2}}  x_{n}^{*} < D$$ 
and
$$D < (\frac{a_n m_n}{\gamma_n (S^{*}+a_n)^{2}}  x_{n}^{*}) + (\frac{f_n}{\gamma_{n}}(f_n-qD_n)(1-p))(x_n^* \,{\left(\frac{m_n }{S^*+a_n }-\frac{S^*\,m_n }{{{\left(S^*+a_n \right)}}^2 }\right)})$$
\end{lemma}

\begin{proof}

We will again begin by looking at the Jacobian with the equilibrium plugged in and then looking the eigenvalues of the system.

\begin{equation*}
J(E^{x_n})=\begin{bmatrix}
f_1 - D_{1} & 0 & \dots  & 0 & 0 \\

0 & f_2-D_2 & \ddots & \vdots & 0\\

\vdots & \vdots & \ddots & \vdots & \vdots\\

0 & 0 & \dots & (f_n-qD_n)(1-p) & x_n^* \,{\left(\frac{m_n }{S^*+a_n }-\frac{S^*\,m_n }{{{\left(S^*+a_n \right)}}^2 }\right)}\\

-\frac{f_1 }{\gamma_{1}}   & -\frac{f_2 }{\gamma_{2}}  & \dots & -\frac{f_n}{\gamma_{n}} & -D + \frac{a_n m_n}{\gamma_n (S^{*}+a_n)^{2}}  x_{n}^{*}
\end{bmatrix}
\end{equation*}

Characteristic polynomial:

$$\rho(\lambda) = (f_1-D_1-\lambda)(f_2-D_2-\lambda)...(f_{n-1}-D_{n-1}-\lambda) \det \left( A\right)$$
with $$A = \begin{bmatrix} (f_n-qD_n)(1-p) - \lambda & x_n^* \,{\left(\frac{m_n }{S^*+a_n }-\frac{S^*\,m_n }{{{\left(S^*+a_n \right)}}^2 }\right)}\\
-\frac{f_n}{\gamma_{n}} & -D + \frac{a_n m_n}{\gamma_n (S^{*}+a_n)^{2}}  x_{n}^{*} - \lambda
\end{bmatrix}$$

Looking at the eigenvalues for this system, the first $n-1$ are of the form $\lambda_i=f_i-D_i$, which by construction must be less than or equal to zero. 
In the case where one or more of them are positive, then the differentials of that $x_i$ equation would be positive and have growth thus $E^{x_{n}}$ would not be an equilibrium point.

The remaining two eigenvalues are determined from the bottom right 2x2 matrix, and to achieve the condition of local sink, we need both eigenvalues to be negative.

When dealing with the final 2x2 matrix we can determine the signs of the eigenvalues using the trace and determinant. 
The determinant is product of the eigenvalues and when positive either means that both eigenvalues are positive or both are negative.
The trace is the sum of the eigenvalues, if the trace is negative coupled with the positive determinant implies that both eigenvalues are negative as desired.

This trace condition holds when:

$$ (f_n-qD_n)(1-p)-\frac{f_n}{\gamma_{n}} -D + \frac{a_n m_n}{\gamma_n (S^{*}+a_n)^{2}}  x_{n}^{*}<0$$

$$ (f_n-qD_n)(1-p)-\frac{f_n}{\gamma_{n}}  + \frac{a_n m_n}{\gamma_n (S^{*}+a_n)^{2}}  x_{n}^{*}<D$$

The determinant condition holds when:

$$(f_n-qD_n)(1-p)(-D + \frac{a_n m_n}{\gamma_n (S^{*}+a_n)^{2}}  x_{n}^{*}) - (-\frac{f_n}{\gamma_{n}})(x_n^* \,{\left(\frac{m_n }{S^*+a_n }-\frac{S^*\,m_n }{{{\left(S^*+a_n \right)}}^2 }\right)}) > 0$$

$$(\frac{a_n m_n}{\gamma_n (S^{*}+a_n)^{2}}  x_{n}^{*}) + (\frac{f_n}{\gamma_{n}}(f_n-qD_n)(1-p))(x_n^* \,{\left(\frac{m_n }{S^*+a_n }-\frac{S^*\,m_n }{{{\left(S^*+a_n \right)}}^2 }\right)}) > D$$
  
\end{proof}

Thus, by fixing all other parameter values we can some amount of command on convergence to this equilibrium point through the controllable resource dilution rate. 
Later on we will showcase how shifting of this parameter can widen or narrow the field of initial conditions that will converge to this $E^{x_{n}}$.

\begin{corollary}
The region of local stability is not impacted by the values of $\gamma_{i}$ $i=\{1,...,n-1\}$, only the speed in convergence. 
\end{corollary}

\begin{proof}
    Intuitively by decreasing the values of $\gamma_{i}$ $i=\{1,...,n-1\}$ the consumption rate of the substrate is increasing and will drive the substrate towards zero faster. 
    This will speed up the results as if the conditions of local stability are met then $x_{n}$ has the lowest break-even point and will be the first species that can be support in the system and drive the others to extinction.
\end{proof}

Using the above results, we can state.

\begin{theorem}
Consider equations ~\ref{eq:chemostat2-1}, \ref{eq:chemostat2-2}, and \ref{eq:chemostat2-3}.
There exists parameters such that for $p=1$, $x_{1}$ competitively excludes $x_{2},x_{3},...,x_{n}$. 
For the same choice of parameters, and suitably chosen $p \in (0,1)$, $x_{n}$ is instead able to competitively exclude all other competitors, for certain initial conditions.
\end{theorem}

\subsection{Numerical Results}
In this section, we examine the numerical results of the new Chemostat system of equations under various parameters and initial conditions.

\begin{figure}[H]
\begin{subfigure}[b]{.475\linewidth}
    \includegraphics[width=3.0in]{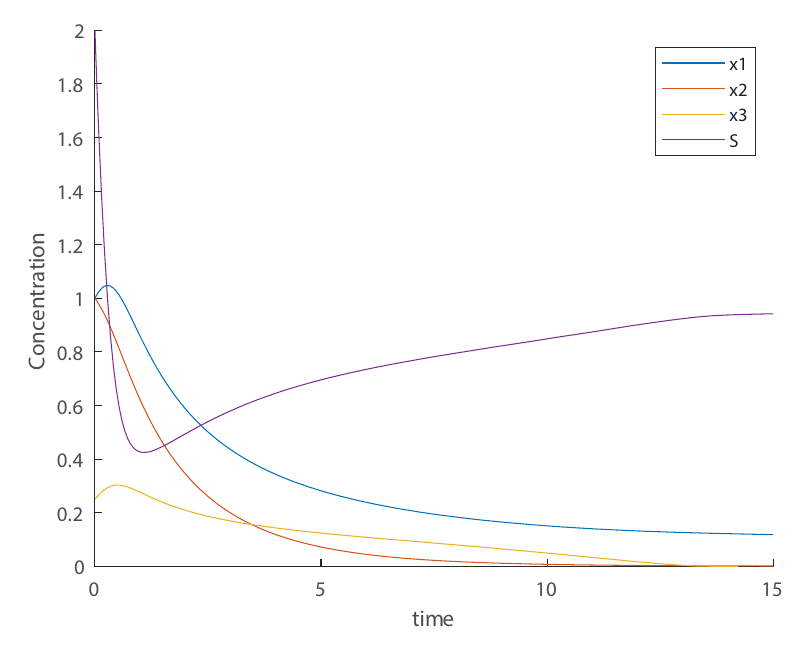}
    \caption{Initial conditions with $x_{3}(0)=.25$}
    \label{subfig:x3lose}
\end{subfigure}
\hfill
\begin{subfigure}[b]{.475\linewidth}
    \includegraphics[width=3.0in]{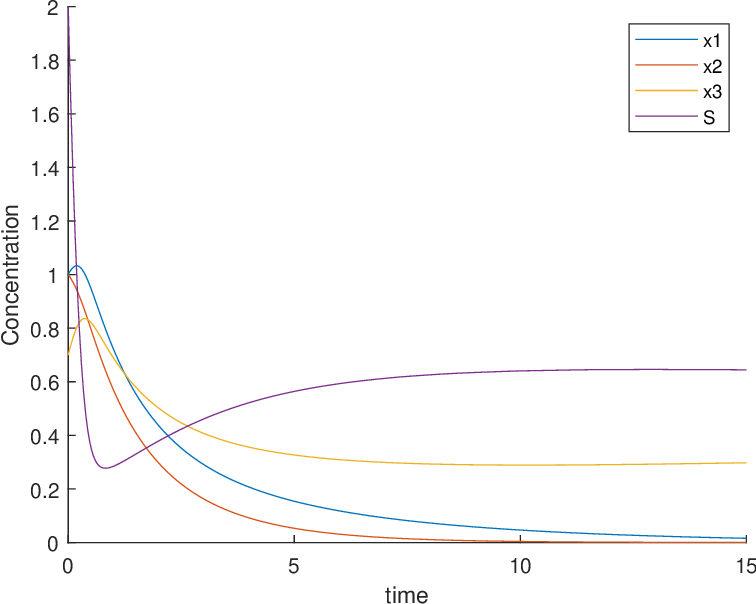}
    \caption{Initial conditions with $x_{3}(0)=1$}
    \label{subfig:x3wins}
\end{subfigure}
\caption{Numerical simulation of where the only change is in terms of initial conditions of $x_3$. Parameters $a_1=1, a_2=1, a_3=1, m_1=2, m_2=1, m_3=3, D_1=1, D_2=.9, D_3=.8, k=.2, p=.5, D=2, S^{0}=1, \gamma_1=1, \gamma_2=1, \gamma_3=.5$, with initial conditions $x_1(0)=1, x_2(0)=1, S(0)=2$.}
\label{fig:pde_ce_sim3}
\end{figure}

We can see that with constant parameter values with $p\neq 1$, that by increasing the initial amount of $x_{3}$ we can go from $x_{1}$ competitively excluding $x_{2}$ and $x_{3}$ to the reversal where $x_{3}$ is the sole survivor.

\begin{figure}[H]
\begin{subfigure}[b]{.475\linewidth}
    \includegraphics[width=3.0in]{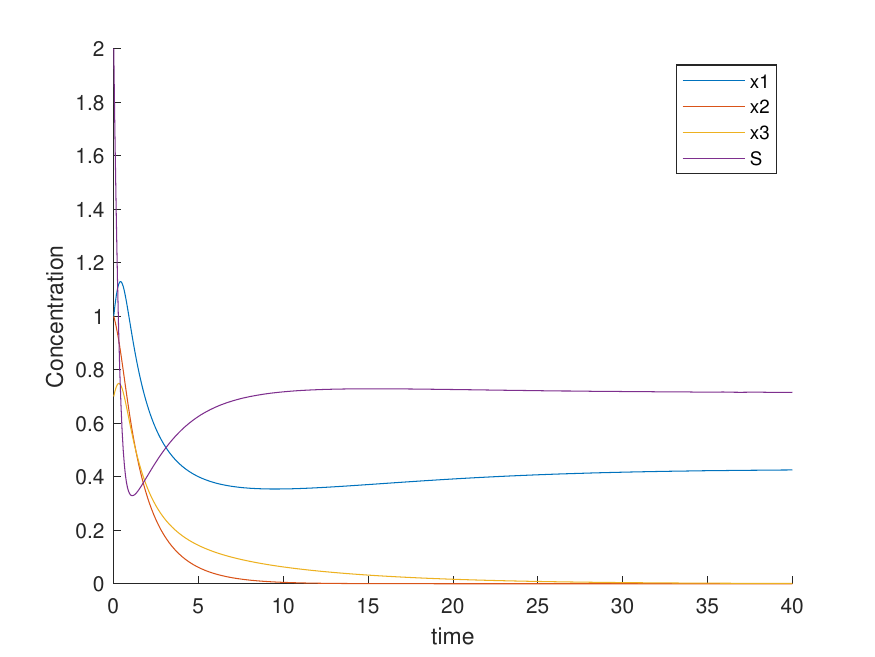}
    \caption{q=1}
    \label{subfig:q1}
\end{subfigure}
\hfill
\begin{subfigure}[b]{.475\linewidth}
    \includegraphics[width=3.0in]{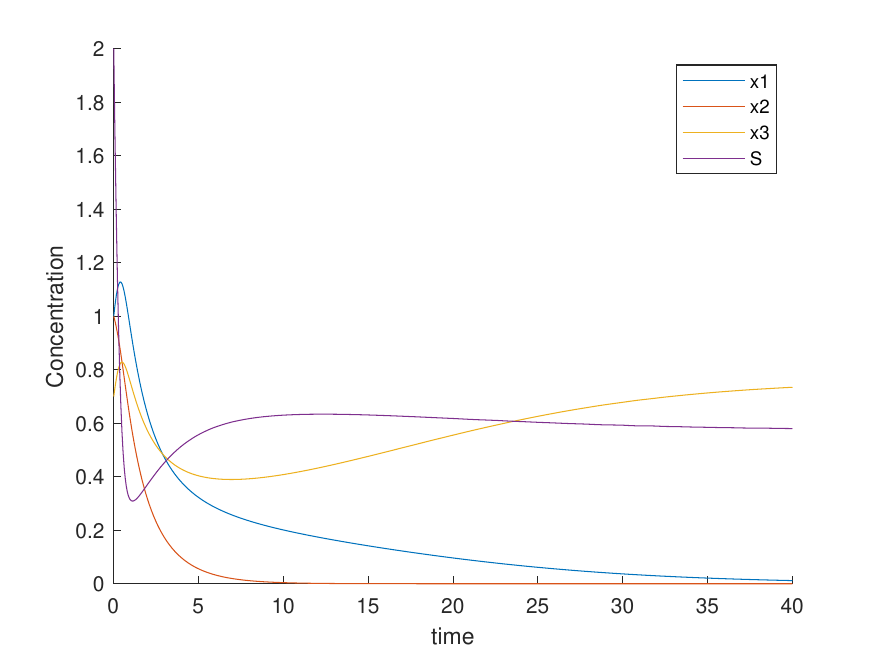}
    \caption{q=.7}
    \label{subfig:q.7}
\end{subfigure}
\caption{Numerical simulation of where the only change is in terms of q parameter. Parameters $a_1=1, a_2=1, a_3=1, m_1=2.4, m_2=1, m_3=2.3, D_1=1, D_2=.9, D_3=1.1, k=.2, p=.5, D=1.5, S^{0}=1, \gamma_1=1, \gamma_2=1, \gamma_3=.5$, with initial conditions $x_1(0)=1, x_2(0)=1, x_3(0)=.7, S(0)=2$.}
\end{figure}

\begin{figure}[H]
\begin{subfigure}[b]{.475\linewidth}
    \includegraphics[width=3.0in]{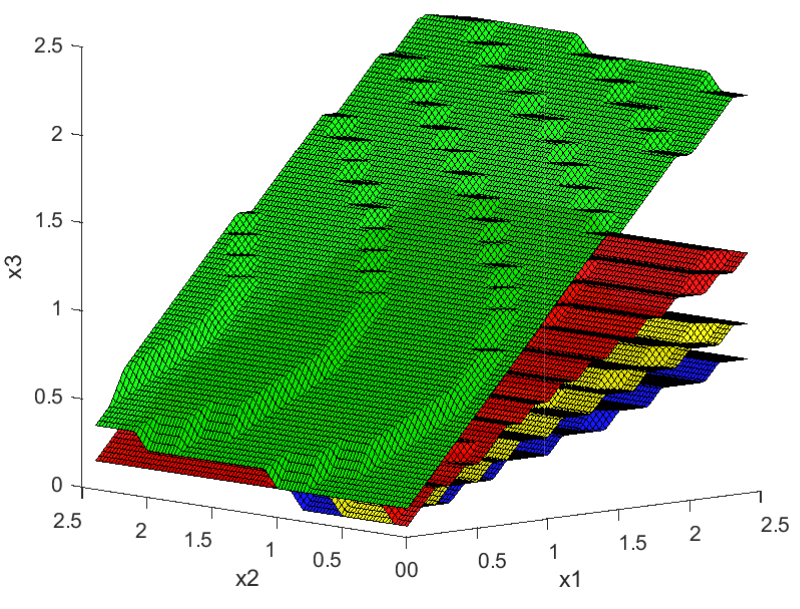}
    \label{ChangeD}
\end{subfigure}
\hfill
\begin{subfigure}[b]{.475\linewidth}
    \includegraphics[width=3.0in]{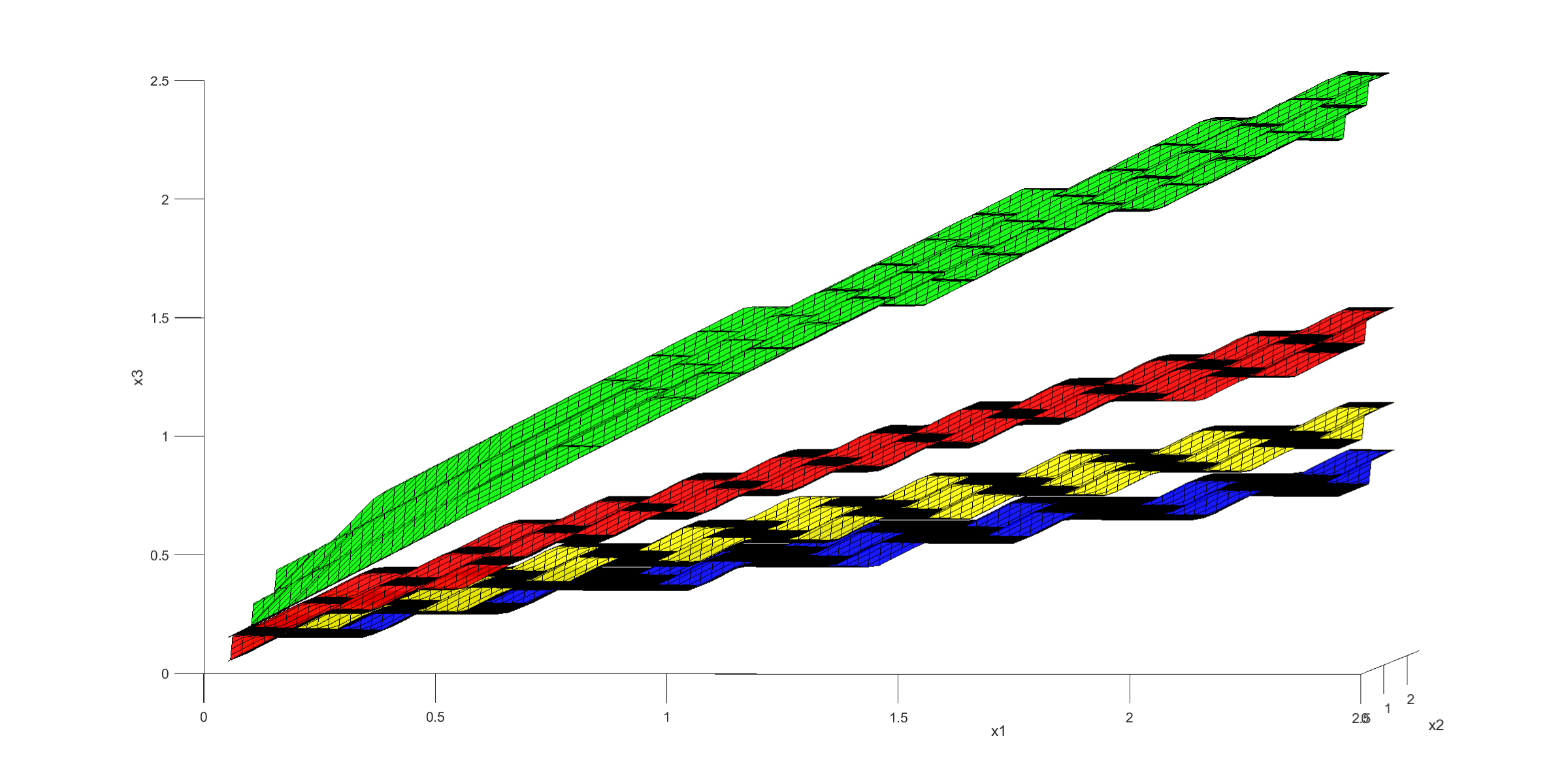}
    \label{Dside}
\end{subfigure}
\caption{Numerical simulation of changing D parameter separatrix from two angles. Parameters $a_1=1, a_2=1, a_3=1, m_1=2.4, m_2=1, m_3=2.3, D_1=1, D_2=.9, D_3=1.1, k=.2, p=.5, S^{0}=1, \gamma_1=1, \gamma_2=1, \gamma_3=.5, S(0)=2$.}
\end{figure}

The figure's colored planes separate the regions of initial conditions where species $x_{1}$ wins (below the planes) and $x_{3}$ wins (above the planes) for several different valuations of $D$. 
The colors of planes correspond to $D=1$ for blue, $D=1.5$ for yellow, $D=2$ for red, and $D=2.5$ for green. 
Recall that $D_{i}=D+$ some death rate for each species, and as D reaches 3 in this simulation then all species die at a rate faster than reproduction and only the substrate will remain. 
The main take-away is that the human controlled dilution rate can directly impacted the percentage of initial conditions that converge to either $x_{1}$ or $x_{3}$.
    \section{Conclusion}

In this paper we discussed the dynamics of a modified extraction chemostat model. 
On the classical model all species  are competing over a finite resource for survival which leads to one species competitively excluding the remaining species \cite{hardin1960competitive} \cite{liu2014competitive}. 
Our model differs from the classical in that a percentage of the weakest/least-fit species was harvested from the environment, while one might believe that this removal will only hurt the species, it in fact led to local convergence where certain initial conditions sees the weaker species out compete its stronger competitors.
The conditions that allowed for this type of dynamics were quantified and shown through several numerical simulations.
We derived an extended Lotka-Volterra competition model from our modified chemostat model that was two-dimensional, and it allowed us to study the system dynamics in the phase plane using nullclines.
We showed that the size of the region of local stability could be increased or decreased by changing the human controlled parameter in the model. 

Several questions remain open at this juncture. One can investigate the spatially explicit case similarly, where the harvesting would be modeled into the boundary conditions. Most works on spatially explicit chemostat models assume equal diffusivity of the species \cite{shi2020coexistence}. Removing this assumption in lieu of the harvesting could enable interesting dynamics. Also, the spatially explicit/PDE case has several connections to surface chemical reaction network theory \cite{clamons2020programming}, which are being currently explored. Also, one can look at the ODE patch case (two patch and multiple patch) with the harvesting occurring in only select patches. We conjecture one might be able to extract similar dynamics for a ``minimal" patch size or number, but this remains to be proven rigorously or perhaps simulated. 

Lastly, the results derived herein would also be interesting to experimentalists and practitioners, who aim to engineer alternative states, via turbidostat type systems, for various biological/chemical processes. The current work provides a theoretical backdrop for such experiments with density dependent harvesting.

    \noindent {\bf Acknowledgments.}
    This work was funded in part by National Science Foundation grant 1900716.

    \bibliography{chemo}
\end{document}